\documentclass[12pt]{amsart}

\usepackage{amsmath}
\usepackage{amssymb}
\usepackage{amsthm}
\usepackage{mathrsfs}
\usepackage{tikz}
\usetikzlibrary{arrows.meta}
\usepackage{hyperref}

\theoremstyle{plain}
\newtheorem{theorem}{Theorem}[section]
\newtheorem{proposition}[theorem]{Proposition}
\newtheorem{lemma}[theorem]{Lemma}
\newtheorem{corollary}[theorem]{Corollary}
\theoremstyle{remark}
\newtheorem{remark}[theorem]{Remark}
\newtheorem*{theoremA}{Theorem A}
\newtheorem*{corollaryB}{Corollary B}

\newcommand{\g}{\mathfrak{g}}
\newcommand{\Uq}{U_q(\g)}
\newcommand{\wt}{\operatorname{wt}}
\newcommand{\ot}{\otimes}
\newcommand{\vp}{\varpi}
\newcommand{\Rhat}{\widehat{R}}
\newcommand{\Qp}{Q_{+}}
\newcommand{\Phis}{\Phi_{s}}
\newcommand{\Pis}{\Pi_{s}}
\newcommand{\K}{\mathbb{K}}

\newcommand{\dnodeS}[1]{\filldraw[fill=black] (#1) circle (2.2pt)}   % short simple root
\newcommand{\dnodeL}[1]{\filldraw[fill=white] (#1) circle (2.2pt)}   % long simple root

\begin{document}
\title{Generalized quantum minors generate quantized coordinate rings}
\author{Ayan Dey}
\date{}

.

\footnotetext{\emph{E-mail addresses}: \texttt{studentayandey@gmail.com},
\texttt{ayan22r@isid.ac.in}.}
\footnotetext{
Indian Statistical Institute, Delhi Center, New Delhi, India}

\begin{abstract}
Let $G$ be a simply connected simple complex algebraic group. It is proved by Oya, Qin, and Yakimov %\cite{OQY} 
that the quantized coordinate ring $\mathcal{O}_q(G)$ is generated
by generalized quantum minors, and therefore carries a quantized cluster algebra structure, for all $G$ but type $F_4$. In this article, we settle the $F_4$ case by an argument uniform across $G_2$, $F_4$, and $E_8$. The main idea is to bootstrap the existing proof in type $E_8$, which relies on Lusztig's canonical basis of the quantum adjoint representation, and replace it with the crystal combinatorics of the quasi-minuscule representation. As a consequence, we prove that $\mathcal{O}_q(F_4)$ also has a quantized cluster algebra structure.
\end{abstract}

\maketitle

\section{Introduction}\label{intro}
Let $G$ be a simply connected simple complex algebraic group. In the paper \cite{FZ}, Fomin and Zelevinsky conjectured that the coordinate rings of various objects related to $G$ carry natural cluster algebra structures, and this has been verified in a wide range of cases over the past few decades. For the coordinate ring $\mathcal{O}(G)$, the expected statement is $\mathcal{O}(G)=\mathscr{A}=\mathscr{U}$,
where $\mathscr{A}$ is the cluster algebra attached to the
Berenstein-Fomin-Zelevinsky seed of the open double Bruhat cell in $G$, and $\mathscr{U}$ is the corresponding upper cluster algebra. Here, we do not add the inverses of the frozen variables in the definitions of $\mathscr{A}$ and $\mathscr{U}$. This was established by Oya \cite{Oya} for every $G$ but type $F_4$. The proof is based on showing that $\mathcal{O}(G)$ is generated by generalized minors, and the equalities follow from this and some other results. For $\g$ of the types: $A_n$--$D_n$, $E_6$, and $E_7$, this generation property was known because in these types, $\g$ has non-trivial minuscule modules as tensor generators. Therefore, the main contribution of \cite{Oya} was to establish the generation theorem for $G_2$ and $E_8$. \cite[Remark~3.10]{Oya} also explains why these techniques do not reach $F_4$.

Quantized cluster algebras were introduced by
Berenstein and Zelevinsky \cite{BZ}. In a recent breakthrough paper, Oya, Qin, and Yakimov \cite{OQY} prove that
the quantized coordinate ring $\mathcal{O}_q(G)$ is generated by generalized quantum minors when $G$ is not of type $F_4$. Using this, combined with the quantum Laurent phenomenon \cite{BZ} and with the coincidence of
$\mathcal{O}_q(G)$ and the quantum upper cluster algebra previously established by Qin and
Yakimov \cite{QY}, one concludes that $\mathcal{O}_q(G)$ admits a quantized cluster algebra
structure. Types $A_n$--$D_n$, $E_6$, and $E_7$ are treated uniformly, type $G_2$ is settled by a direct brute force
computation \cite[Theorem~5.7]{OQY}, and type $E_8$ by combining Lusztig's canonical basis of the quantum adjoint representation \cite{Lusztig} with $\hat{R}$-matrix techniques \cite[Theorem~5.9]{OQY}. In fact, except for the groups $F_4$ and $E_8$, they have proved this generation theorem over much smaller ground rings than $\mathbb{Q}(q^{1/2})$ (see remark \ref{srings}). 

Therefore, the type $F_4$ is the only case left open in both classical and quantum settings. The natural guess would be to work out the classical case first because in the quantum case, certain computations while decomposing the tensor product of modules become much more cumbersome, as observed in \cite[Remarks 5.8, and 5.12] {OQY}. But in type $F_4$, as we observe later on, the situation is quite the opposite, and the classical problem seems more complicated than its quantum analog. The purpose of this article is to settle the quantum $F_4$ case. Our main
result is the following.
\begin{theoremA}[= Theorem \ref{thm:main}, Corollary \ref{cor:allG}]
Let $\g$ be of type $G_2$, $F_4$ or $E_8$ and $V=V(\vp)$ be the quasi-minuscule module with $\vp$ the highest short root. Then every matrix coefficient of $V$ involving the zero-weight vectors or functionals can be written in terms of matrix coefficients of $V$ involving weight vectors and functionals of non-zero weight. Consequently, the quantized coordinate ring $\mathcal{O}_q(G)$ is generated by generalized quantum minors over $\mathbb{Q}(q)$.
\end{theoremA}
In these three types, the quasi-minuscule module $V(\vp)$ works as a tensor generator of the representation ring. The proof of Theorem A is uniform across these three types and thus reproduces the results of \cite[Theorems~5.7 and~5.9]{OQY} for $G_2$ and $E_8$. When $\g$ is simply-laced, $V(\vp)$ is the quantum adjoint representation. For $G_2$ and $F_4$, $V(\vp)$ is no longer the quantum adjoint representation. But the underlying crystal basis $B(\vp)$  admits some common properties independent of the simply-laced condition. We study the crystal combinatorics of $B(\vp)\ot B(\vp)$  to isolate the copy of $B(\vp)$ inside the product crystal and control the weights of both tensor factors at every node (Proposition \ref{prop:twosided}). We then lift these properties to the module level using the uniqueness of crystal bases (Theorem \ref{thm:Tminus}). This weight control, together with a braiding twist (Proposition \ref{prop:d}, Corollary \ref{cor:extract}), is precisely the ingredient needed to prove Theorem \ref{thm:main}. We immediately obtain the following Corollary.

\begin{corollaryB}[= Corollary~\ref{cor:cluster}]
Let $G$ be of type $F_4$. Then $\mathcal{O}_q(G)$ admits a quantized cluster algebra
structure over $\mathbb{Q}(q^{1/2})$.
\end{corollaryB}
%Next, let us give a brief summary of the content. Section \ref{prelim} fixes notation for the quantized enveloping algebra, recalls crystal bases and the tensor product rule, and introduces quantum coordinate rings and generalized quantum minors. Section \ref{qm} describes the quasi-minuscule module $V(\vp)$, its crystal basis,
%and focuses on the three types $G_2$, $F_4$, and $E_8$. Section \ref{main} contains the crystal-theoretic heart of the argument. Section \ref{minors} contains the main
%results. 
Finally, we describe (see remark \ref{obs}) why our present proof in the quantum setting fails to specialize at $q=1$, identify the obstruction, and note that the problem in the classical $F_4$ case remains open. 

\section{Preliminaries}\label{prelim}

\subsection{The quantized enveloping algebra}\label{ss:qea}

Let $\g$ be a finite-dimensional complex simple Lie algebra with Cartan subalgebra
$\mathfrak h$, simple roots $\Pi:=\{\alpha_j\}_{j\in I}$, root system $\Phi\subseteq\mathfrak
h^*$ and set of positive roots $\Phi_+$. We write $Q=\bigoplus_{j\in I}\mathbb
Z\alpha_j$ for the root lattice and $\Qp=\sum_{j\in I}\mathbb Z_{\ge0}\alpha_j$ for the
positive root lattice. Let $P$ be the weight lattice, $P_+$ the dominant chamber,
$\{\vp_j\}_{j\in I}$ the fundamental weights. $\rho:=\tfrac12\sum_{\alpha\in\Phi_+}\alpha$ stands for 
the Weyl vector. The Weyl group is denoted by $W$, and $w_0\in W$ is the longest element. $P$ is also equipped with the partial order: $\nu\ge\nu'$ when
$\nu-\nu'\in\Qp$.

There is a unique $W$-invariant symmetric bilinear form $(\,,\,)$ on $\mathbb{R}\Phi$ normalised by
$(\alpha,\alpha)=2$ on short roots $\alpha$. We write $d_j=(\alpha_j,\alpha_j)/2$, and
$q_j=q^{d_j}$. $\alpha_j^\vee=2\alpha_j/(\alpha_j,\alpha_j)$ are the coroots of
$\alpha_j$, and $\langle\lambda,\alpha_j^\vee\rangle
:=\frac{2(\lambda,\alpha_j)}{(\alpha_j,\alpha_j)}$ for any $\lambda\in P$.

Let $\Uq$ be the quantized universal enveloping algebra (QUEA) of $\g$
over $\mathbb{K}:=\mathbb Q(q)$ with $q$ transcendental, generated by $E_j,F_j,K_j^{\pm1}$ $(j\in I)$ subject to the standard relations.
% $\Uq$ has a triangular
%decomposition $\Uq\ \cong\ U^-\ot U^0\ot U^+$ as a $\K$-vector space, where $U^+$ and $U^-$ are the subalgebras generated by the
%$E_j$ and by the $F_j$ respectively, and $U^0$ is generated by the $K_j^{\pm1}$. Then
%$U^\pm$ carries a $\Qp$-grading:
%$U^+=\bigoplus_{\eta\in\Qp}U^+_\eta,\; U^-=\bigoplus_{\eta\in\Qp}U^-_{-\eta},
%\; E_j\in U^+_{\alpha_j},\; F_j\in U^-_{-\alpha_j}$. 
The Hopf structure on $\Uq$ is defined by,
\begin{equation}\label{eq:hopf}
\Delta(E_j)=E_j\ot1+K_j\ot E_j,\qquad
\Delta(F_j)=F_j\ot K_j^{-1}+1\ot F_j,\qquad
\Delta(K_j)=K_j\ot K_j,
\end{equation}
$\varepsilon(E_j)=\varepsilon(F_j)=0$, $\varepsilon(K_j)=1$,
$S(E_j)=-K_j^{-1}E_j$, $S(F_j)=-F_jK_j$ and $S(K_j)=K_j^{-1}$. 

A finite dimensional $\Uq$-module $M$ is of \emph{type $1$} if $M=\bigoplus_{\nu\in P}M_\nu$, where
\[
M_\nu:=\{\,u\in M\mid K_j\,u=q^{(\nu,\alpha_j)}u\ \text{ for all }j\in I\,\}.
\]
%The graded pieces of $U^\pm$ then shift weights, $U^+_\eta M_\nu\subseteq M_{\nu+\eta}$
%and $U^-_{-\eta}M_\nu\subseteq M_{\nu-\eta}$ for $\eta\in\Qp$. 
The dual $M^*$ is also equipped with a $\Uq$-module structure, defined using the antipode $S$
by $\langle x.\xi,u\rangle:=\langle\xi,S(x)u\rangle$ for $x\in\Uq$, $\xi\in M^*$ and
$u\in M$. The category of finite-dimensional type-$1$
$\Uq$-modules is semisimple; its simple objects are the modules $V(\lambda)$ of highest
weight $\lambda\in P_+$. The weight spaces $V(\lambda)_{w\lambda}$, for any $w\in W$, are called the extremal weight
spaces, and these are one-dimensional. Furthermore, the unique lowest weight of $V(\lambda)$ is $w_0\lambda$.

\subsection{Crystal bases and the tensor product rule}\label{ss:crystal}

Let $A_0\subseteq\K$ denote the ring of rational functions in $q$ that are regular at
$q=0$. It is a local ring with maximal ideal $qA_0$.
$A_0^\times:=\{f\in A_0\mid f(0)\neq0\}$, in particular $1+qA_0\subseteq A_0^\times$. On a finite dimensional type-$1$ module, we assume the lower Kashiwara operators $\tilde e_j,\tilde f_j$
$(j\in I)$ act as in \cite{Kashiwara}. Fix $\lambda\in P_+$ and a highest weight
vector $v_\lambda \in V(\lambda)$, define
\[
L(\lambda):= \sum_{j_1,\dots,j_\ell\in I}
A_0\,\tilde f_{j_1}\cdots\tilde f_{j_\ell}\,v_\lambda\ \subseteq\ V(\lambda),
\]
\[
B(\lambda):=\big\{\,\tilde f_{j_1}\cdots\tilde f_{j_\ell}\,v_\lambda
\ \ \mathrm{mod}\ qL(\lambda)\ \big|\ j_1,\dots,j_\ell\in I\,\big\}
\ \smallsetminus\ \{0\}\ \subseteq\ L(\lambda)/qL(\lambda).
\]
By \cite{Kashiwara}, the pair $(L(\lambda),B(\lambda))$ is the \emph{lower crystal base}
of $V(\lambda)$. We call the elements of $B(\lambda)$
nodes, and we write $\wt b$ for its weight. Since $\g$ is of finite type, $B(\lambda)$ has a unique lowest weight node $b_{w_0\lambda}$ annihilated by all the $\tilde f_j$, and every node of $B(\lambda)$ is reachable from
$b_{w_0\lambda}$ using only the raising operators $\tilde e_j$. In other words, $B(\lambda)=\big\{\,\tilde e_{j_1}\cdots\tilde e_{j_\ell}\,b_{w_0\lambda}
\ \big|\ j_1,\dots,j_\ell\in I\,\big\}\ \smallsetminus\ \{0\}$.

To define the tensor product crystal, we need a slightly different coproduct on $\Uq$.
\begin{equation}\label{eq:deltaminus}
\Delta_-(E_j)=E_j\ot K_j^{-1}+1\ot E_j,\qquad
\Delta_-(F_j)=F_j\ot1+K_j\ot F_j,\qquad
\Delta_-(K_j)=K_j\ot K_j.
\end{equation}
For two type-1 $\Uq$-modules $V,W$ denote by $V\ot_- W$ and $V\ot_+ W$ the tensor products formed using
\eqref{eq:deltaminus} and \eqref{eq:hopf}. Set
$\varepsilon_j(b)=\max\{n:\tilde e_j^{\,n}b\neq0\}$ and
$\varphi_j(b)=\max\{n:\tilde f_j^{\,n}b\neq0\}$, so that
$\varphi_j(b)-\varepsilon_j(b)=\langle\wt b,\alpha_j^\vee\rangle$. Then, for any 
$\lambda,\lambda'\in P_+$ the pair
$(L(\lambda)\ot L(\lambda'),\;B(\lambda)\ot B(\lambda'))$ is the lower crystal base of
$V(\lambda)\ot_-V(\lambda')$, and
\begin{equation}\label{eq:tensorrule}
\tilde e_j(b_1\ot b_2)=
\begin{cases}\tilde e_j b_1\ot b_2,&\varphi_j(b_1)\ge\varepsilon_j(b_2),\\
b_1\ot\tilde e_j b_2,&\varphi_j(b_1)<\varepsilon_j(b_2),\end{cases}
\qquad
\tilde f_j(b_1\ot b_2)=
\begin{cases}\tilde f_j b_1\ot b_2,&\varphi_j(b_1)>\varepsilon_j(b_2),\\
b_1\ot\tilde f_j b_2,&\varphi_j(b_1)\le\varepsilon_j(b_2).\end{cases}
\end{equation}

\subsection{Quantum coordinate rings and generalized quantum minors}\label{ss:oq}
For a finite dimensional type-1 $\Uq$-module $U$ and $u\in U$, $\xi\in U^*$, we denote
$c^{U}(\xi,u)\in\Uq^*$ for the matrix coefficient $c^U(\xi,u)(X):=\langle\xi,X.u
\rangle$. Let $\mathcal{O}_q(G)\subseteq\Uq^*$ be the \emph{quantized coordinate ring} spanned by matrix coefficients of all finite-dimensional type-1 $\Uq$-modules. Let $\lambda\in P_+$ and $w_1,w_2\in W$. Fix a non-zero vector $v_{w_2\lambda}$ spanning $V(\lambda)_{w_2\lambda}$, and let
$\xi_{w_1\lambda}\in(V(\lambda)^*)_{-w_1\lambda}$ be the dual functional of the non-zero vector $v_{w_1\lambda}\in V(\lambda)_{w_1\lambda}$. Then the associated
\emph{generalized quantum minor} is defined by
\[
\Delta_{w_1\lambda,\,w_2\lambda}:=c^{V(\lambda)}\big(\xi_{w_1\lambda},\,v_{w_2\lambda}\big)
\ \in\ \mathcal{O}_q(G),
\]
and it is unique up to a non-zero scalar. For finite dimensional type-1 $\Uq$-modules $U_1$ and $U_2$
there is a $\Uq$-module isomorphism $U_2^*\ot U_1^*\ \xrightarrow{\ \sim\ }\ (U_1\ot U_2)^*$ by assigning $\xi_2\otimes \xi_1\to [u_1\otimes u_2 \to \xi_1(u_1).\xi_2(u_2)]$
under which
\begin{equation}\label{eq:mult}
c^{U_1\ot U_2}(\xi_2\ot\xi_1,\ u_1\ot u_2)=c^{U_1}(\xi_1,u_1)\,c^{U_2}(\xi_2,u_2)
\end{equation}
for all $u_1\in U_1$, $u_2\in U_2$, $\xi_2\in U_2^*$ and $\xi_1\in U_1^*$.

\section{Quasi-minuscule modules}\label{qm}
A non-zero dominant weight $\vp \in P_+$ is \emph{minuscule} if every weight of $V(\vp)$ lies in the Weyl group orbit of $\vp$. Similarly, a non-zero dominant weight $\vp \in P_+$ is called \emph{quasi-minuscule} if it is either minuscule or the only weight of $V(\vp)$ not in the orbit $W\cdot\vp$ is the zero weight. Each simple Lie algebra admits a unique quasi-minuscule highest weight module that is not minuscule with the highest weight = the highest short root of the associated root system. 
\subsection{The general construction}\label{ss:qmgen}
Let $\Phis\subseteq\Phi$ be the set
of all short roots, and let
$\Pis:=\Pi\cap\Phis$ be the set of all short simple roots. The set $\Phis$ forms a single
$W$-orbit containing a unique dominant weight $\vp$, the highest short root. The
quasi-minuscule module $V:=V(\vp)$ admits a weight basis
\begin{equation}\label{eq:qmbasis}
\{v_\mu\mid\mu\in\Phis\}\ \sqcup\ \{z_{\alpha_k}\mid\alpha_k\in\Pis\},
\qquad v_\mu\in V_\mu,\quad z_{\alpha_k}\in V_0,
\end{equation}
so that the non-zero weights of $V$ are all extremal weights, each of multiplicity one. $\dim V_0=|\Pis|$. Write $[2]_{q_j}=q_j+q_j^{-1}$. The
$\Uq$-module action on $V(\vp)$ is defined as follows \cite[5A.2]{Jantzen}.

\noindent\emph{(Q1)} For all $j\in I$,
\[
K_j\,v_\mu=q^{(\mu,\alpha_j)}v_\mu,\qquad K_j\,z_{\alpha_k}=z_{\alpha_k}.
\]

\noindent\emph{(Q2)} For $\mu\in\Phis$ with
$\mu\neq\pm\alpha_j$ one has $\langle\mu,\alpha_j^\vee\rangle\in\{0,\pm1\}$, and
\[
E_j\,v_\mu=
\begin{cases}v_{\mu+\alpha_j},&\langle\mu,\alpha_j^\vee\rangle=-1,\\ 0,&\text{else},\end{cases}
\qquad
F_j\,v_\mu=
\begin{cases}v_{\mu-\alpha_j},&\langle\mu,\alpha_j^\vee\rangle=1,\\ 0,&\text{else}.\end{cases}
\]

\noindent\emph{(Q3)} If $\alpha_j\notin\Pis$, then
$E_j\,z_{\alpha_k}=F_j\,z_{\alpha_k}=0$ for all $\alpha_k\in\Pis$.

\noindent\emph{(Q4)} If $\alpha_j\in\Pis$, then
\[
\begin{aligned}
&E_j\,v_{-\alpha_j}=z_{\alpha_j}, &\quad &E_j\,z_{\alpha_j}=[2]_{q_j}v_{\alpha_j}, &\quad &E_j\,v_{\alpha_j}=0,\\
&F_j\,v_{\alpha_j}=z_{\alpha_j}, &\quad &F_j\,z_{\alpha_j}=[2]_{q_j}v_{-\alpha_j}, &\quad &F_j\,v_{-\alpha_j}=0.
\end{aligned}
\]

\noindent\emph{(Q5)} If $\alpha_j,\alpha_k\in\Pis$ with
$j\neq k$, then
\[
E_j\,z_{\alpha_k}=
\begin{cases}v_{\alpha_j},&\langle\alpha_k,\alpha_j^\vee\rangle=-1,\\ 0,&\text{else},\end{cases}
\qquad
F_j\,z_{\alpha_k}=
\begin{cases}v_{-\alpha_j},&\langle\alpha_k,\alpha_j^\vee\rangle=-1,\\ 0,&\text{else}.\end{cases}
\]

When $\g$ is simply-laced, $V$ is the quantum adjoint representation. In this case, the relations
\emph{(Q1)}--\emph{(Q5)} coincide verbatim with those of \cite[Theorem~5.10]{OQY} and
with Lusztig's description \cite{Lusztig} of the canonical basis of the quantum adjoint
representation.

\begin{proposition}\label{prop:crystal}
The $A_0$-submodule $L(\vp)\ =\ \bigoplus_{\mu\in\Phis}A_0\,v_\mu\ \oplus\ \bigoplus_{\alpha_k\in\Pis}A_0\,z_{\alpha_k}\ \subseteq\ V(\vp)$ is the lower crystal lattice of $V(\vp)$, with the highest weight vector $v_{\vp}$. Its reduction $B(\vp)\ =\ \{\,\overline{v}_\mu\mid\mu\in\Phis\,\}\ \sqcup\
\{\,\overline{z_{\alpha_k}}\mid\alpha_k\in\Pis\,\}\ \subseteq\ L(\vp)/qL(\vp)$
is the lower crystal basis.
\end{proposition}

We write $b_\mu:=\overline v_\mu\ (\mu\in\Phis),\;\;
b_{k,0}:=\overline{z_{\alpha_k}}\ (\alpha_k\in\Pis)$. The highest weight node is $b_\vp$, the lowest weight node is $b_{-\vp}$ (because $w_0.\vp= -\vp$). For any $\alpha_k\in\Pis$, the node $b_{k,0}$ is the \emph{unique}
zero-weight node of $B(\vp)$ for which
\begin{equation}\label{eq:zkstring}
\varepsilon_j(b_{k,0})=\varphi_j(b_{k,0})=\delta_{k,j}\qquad(j\in I).
\end{equation}
For a proof of Proposition~\ref{prop:crystal} and \eqref{eq:zkstring}, see \cite[Lemma 9.6.B]{Jantzen} and the subsequent proof.

\subsection{The three types}\label{ss:threetypes}

From now on, $\g$ (unless otherwise stated) is of type $G_2$, $F_4$, or $E_8$. We assume the Bourbaki convention \cite[Ch. VI, Plates~VII--IX]{Bourbaki} for the Dynkin diagrams below. Short simple roots are denoted by filled nodes, and long simple roots by empty nodes.

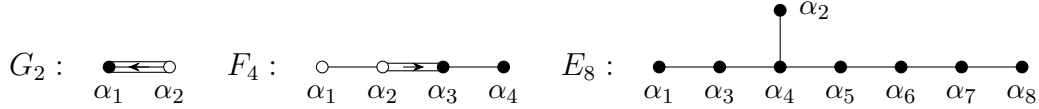
\begin{center}
\begin{tabular}{@{}c@{\quad}c@{\quad}c@{}}
$G_2:$\;
\begin{tikzpicture}[baseline=-3pt]
  \draw (0,0.07)--(0.8,0.07);
  \draw (0,0)--(0.8,0);
  \draw (0,-0.07)--(0.8,-0.07);
  \draw[-{Stealth[length=4.5pt,width=3.5pt]}] (0.53,0)--(0.25,0);
  \dnodeS{0,0}; \dnodeL{0.8,0};
  \node[below=3pt] at (0,0) {\small$\alpha_1$};
  \node[below=3pt] at (0.8,0) {\small$\alpha_2$};
\end{tikzpicture}
&
$F_4:$\;
\begin{tikzpicture}[baseline=-3pt]
  \draw (0,0)--(0.8,0);
  \draw (0.8,0.05)--(1.6,0.05);
  \draw (0.8,-0.05)--(1.6,-0.05);
  \draw[-{Stealth[length=4.5pt,width=3.5pt]}] (1.07,0)--(1.35,0);
  \draw (1.6,0)--(2.4,0);
  \dnodeL{0,0}; \dnodeL{0.8,0}; \dnodeS{1.6,0}; \dnodeS{2.4,0};
  \node[below=3pt] at (0,0) {\small$\alpha_1$};
  \node[below=3pt] at (0.8,0) {\small$\alpha_2$};
  \node[below=3pt] at (1.6,0) {\small$\alpha_3$};
  \node[below=3pt] at (2.4,0) {\small$\alpha_4$};
\end{tikzpicture}
&
$E_8:$\;
\begin{tikzpicture}[baseline=-3pt]
  \draw (0,0)--(4.8,0);
  \draw (1.6,0)--(1.6,0.75);
  \dnodeS{0,0}; \dnodeS{0.8,0}; \dnodeS{1.6,0}; \dnodeS{2.4,0};
  \dnodeS{3.2,0}; \dnodeS{4.0,0}; \dnodeS{4.8,0}; \dnodeS{1.6,0.75};
  \node[below=3pt] at (0,0) {\small$\alpha_1$};
  \node[below=3pt] at (0.8,0) {\small$\alpha_3$};
  \node[below=3pt] at (1.6,0) {\small$\alpha_4$};
  \node[below=3pt] at (2.4,0) {\small$\alpha_5$};
  \node[below=3pt] at (3.2,0) {\small$\alpha_6$};
  \node[below=3pt] at (4.0,0) {\small$\alpha_7$};
  \node[below=3pt] at (4.8,0) {\small$\alpha_8$};
  \node[right=3pt] at (1.6,0.75) {\small$\alpha_2$};
\end{tikzpicture}
\end{tabular}
\end{center}

Under this labeling, $\vp=\vp_i$, where
$i=1\ \ (G_2),\;i=4\ \ (F_4),\;i=8\ \ (E_8)$. We fix this index $i$ once and for all. It again follows that $\alpha_i\in\Pis$, and
$z_{\alpha_i}\in V_0$ together with the zero-weight node
$b_{i,0}=\overline{z_{\alpha_i}}$ of Proposition~\ref{prop:crystal} are well-defined.
Furthermore, $\dim V=7,\ 26,\ 248,\; m:=\dim V_0=|\Pis|=1,\ 2,\ 8$. In each of these three types $-w_0=\mathrm{id}$ and $P=Q$; consequently
$(\lambda,\mu)\in\mathbb Z$ for all $\lambda,\mu\in P$. $V$ appears as a direct summand in $V\ot V$ with multiplicity 1. We use the short-hand notation $[V\ot V:V]=1$ to recall this fact later on (see Lemma \cite[5A.9b]{Jantzen}).

\section{Factorization of weights and the braiding twist}\label{main}

\subsection{The two-sided weight control theorem}\label{ss:twosided}

Applying \eqref{eq:tensorrule} to $B(\vp)\ot B(\vp)$ and using $[V\ot V:V]=1$, there is
a unique connected component $C\cong B(\vp)$ inside $B(\vp)\ot B(\vp)$. The highest and lowest weight nodes of $C$ are of the form
\begin{equation}\label{eq:sourcesink}
x=b_\vp\ot c,\qquad x_{\mathrm{low}}=b'\ot b_{-\vp}
\end{equation}
with $\wt c=\wt b'=0$.
The following proposition generalizes \cite[Lemma~3.1]{Dey}, which establishes that the right tensor factor of every node of $C$ except the highest weight node $x$ has negative weight. Here, we also prove a similar statement for the left tensor factor. Hence, the weights of both factors are controlled simultaneously for every node except $x$ and $x_{low}$ of $C$.

\begin{proposition}\label{prop:twosided}
Let $C\cong B(\vp)$ be the connected component of $B(\vp)\ot B(\vp)$ with $x$
and $x_{\mathrm{low}}$ as in \eqref{eq:sourcesink}. Then $c=b'=b_{i,0}$. Moreover, every node
$y=b_1\ot b_2\in C$ with $y\neq x,\,x_{\mathrm{low}}$ must have $\wt b_1>0$ and $\wt b_2<0$.
\end{proposition}

\begin{proof}
Since $\wt c=0$ and $c\neq b_\vp$, $\varepsilon_j(c)>0$ for some $j$. As $x$ is the highest weight node of $C$, $\varepsilon_j(c)\le\varphi_j(b_\vp)=\langle\vp,\alpha_j^\vee\rangle=\delta_{ji}$ for all
$j$. Hence
\begin{equation}\label{eq:c}
\varepsilon_i(c)=1,\qquad \varepsilon_j(c)=0\quad(j\neq i).
\end{equation}
$\wt c=0$ also implies $\varphi_i(c)=\varepsilon_i(c)+\langle\wt c,\alpha_i^\vee\rangle=1$
and $\varphi_j(c)=0$ for $j\neq i$. Therefore, \eqref{eq:zkstring} identifies
$c=b_{i,0}$. On the other hand, $\wt b'=0$ and $b'\neq b_{-\vp}$ give
$\varphi_j(b')>0$ for some $j$. As $x_{low}$ is the lowest weight node, $\varphi_j(b')\le\varepsilon_j(b_{-\vp})=\delta_{ji}$ for all $j$. Thus, we get
\begin{equation}\label{eq:b}
\varphi_i(b')=1,\qquad \varphi_j(b')=0\quad(j\neq i).
\end{equation}
Similarly, $\varepsilon_i(b')=1$, $\varepsilon_j(b')=0$ $(j\neq i)$, and
\eqref{eq:zkstring} again gives $b'=b_{i,0}=c$.

We first show that $\wt b_2<0$ for every node $y\neq x$. Any such node
$y$ is of the form $y=\tilde f_{j_\ell}\cdots\tilde f_{j_1}x$ for some
$j_1,\dots,j_\ell$ with the intermediate applications nonzero. We induct on
$\ell$. When $j\neq i$, \eqref{eq:c} gives
$\varphi_j(b_\vp)=0=\varepsilon_j(b_{i,0})$, so by \eqref{eq:tensorrule} the operator
$\tilde f_j$ acts on the right factor; but $\varphi_j(b_{i,0})=\varepsilon_j(b_{i,0})
+\langle\wt b_{i,0},\alpha_j^\vee\rangle=0$. So $\tilde f_j b_{i,0}=0$ and $\tilde f_j x=0$. When $j=i$,
$\varphi_i(b_\vp)=1=\varepsilon_i(b_{i,0})$, so $\tilde f_i$ again acts on the right, and
$\varphi_i(b_{i,0})=1$ gives
\[
\tilde f_i x=b_\vp\ot\tilde f_i b_{i,0},\qquad \wt(\tilde f_i b_{i,0})=\wt b_{i,0}-\alpha_i=-\alpha_i<0.
\]
Thus the assertion holds for $\ell=1$. For the inductive step, suppose $y=b_1\ot
b_2$ with $\wt b_2<0$, and let $y'=\tilde f_j y\neq0$. By \eqref{eq:tensorrule},
$\tilde f_j$ acts either on $b_1$, leaving $b_2$ unchanged, or on $b_2$, replacing
it by $\tilde f_j b_2$ of weight $\wt b_2-\alpha_j$. In either case the right
factor of $y'$ has negative weight, completing the induction.

We now show that $\wt b_1>0$ for every $y\neq x_{\mathrm{low}}$. As $C\cong B(\vp)$, any node $y\neq x_{\mathrm{low}}$ is of the form
$y=\tilde e_{j_\ell}\cdots\tilde e_{j_1}x_{\mathrm{low}}$ with the intermediate
applications nonzero. We again induct on $\ell$. For $\ell=1$ we compute $\tilde e_j
x_{\mathrm{low}}$. When $j\neq i$, \eqref{eq:b} gives
$\varphi_j(b_{i,0})=0=\varepsilon_j(b_{-\vp})$, so $\tilde e_j$ acts on the left factor;
but $\varepsilon_j(b_{i,0})=\varphi_j(b_{i,0})-\langle\wt b_{i,0},\alpha_j^\vee\rangle=0$, so $\tilde e_j
b_{i,0}=0$ and $\tilde e_j x_{\mathrm{low}}=0$. When $j=i$,
$\varphi_i(b_{i,0})=1=\varepsilon_i(b_{-\vp})$, so $\tilde e_i$ acts on the left, and
$\varepsilon_i(b_{i,0})=1$ gives
\[
\tilde e_i x_{\mathrm{low}}=(\tilde e_i b_{i,0})\ot b_{-\vp},\qquad
\wt(\tilde e_i b_{i,0})=\wt b_{i,0}+\alpha_i=\alpha_i>0.
\]
For the inductive step, suppose $y=b_1\ot b_2$ with $\wt b_1>0$ and let $y'=\tilde
e_j y\neq0$. By \eqref{eq:tensorrule}, $\tilde e_j$ acts either on $b_1$, replacing
it by $\tilde e_j b_1$ of weight $\wt b_1+\alpha_j$, or on $b_2$, leaving $b_1$
unchanged. In either case, the left factor of $y'$ has positive weight. 

So a node $y\in C\setminus \{x, x_{\mathrm{low}}\}$ must have $\wt b_1>0$ and $\wt b_2<0$.
\end{proof}

\begin{theorem}\label{thm:Tminus}
There is a $\Uq$-morphism $T_-\colon V\ot_- V\to V$ sending the product crystal lattice
onto $L(\vp)$, its crystal limit $\overline{T_-}$ restricts to a bijection
$C\xrightarrow{\sim}B(\vp)$ and annihilates every other component. It has the
following properties.
\begin{enumerate}
\item[\rm(i)] For $\mu\in W\vp$ with $\mu\neq\pm\vp$, the unique node of $C$ of
weight $\mu$ is of the form $y=b_{\mu_1}\ot b_{\mu_2}$ with $\mu_1,\mu_2\in\Phis$,
$\mu_1>0>\mu_2$ and $\mu_1+\mu_2=\mu$. Moreover,
\[
T_-(v_{\mu_1}\ot v_{\mu_2})=c_\mu\,v_\mu,\qquad c_\mu\in A_0^\times.
\]
\item[\rm(ii)] The $m$ zero-weight nodes of $C$ are $b_{\nu_k}\ot b_{-\nu_k}$,
$k=1,2,\dots,m$, with distinct short roots $\nu_k>0$, and the vectors
$\{T_-(v_{\nu_k}\ot v_{-\nu_k})\}_{k=1}^m$ form a basis of $V_0$.
\end{enumerate}
\end{theorem}

\begin{proof}
By the uniqueness of crystal bases \cite[Theorem 3]{Kashiwara} there is a
lattice-preserving $\Uq$-module isomorphism $S\colon V\ot_- V\xrightarrow{\sim}\bigoplus_k
V(\lambda_k)$ whose crystal limit identifies $B(\vp)\ot B(\vp)$ with
$\bigsqcup_k B(\lambda_k)$. As $[V\ot V:V]=1$, $V\ot V$ has exactly one summand isomorphic to $V(\vp)$ with crystal
$C\cong B(\vp)$. Write $T_-=\pi_{V(\vp)}\circ S$.

For (i), let $y=b_1\ot b_2$ be the unique node of $C$ of weight $\mu$. By
Proposition~\ref{prop:twosided}, $\wt b_1>0$ and $\wt b_2<0$; since the non-zero weights
of $B(\vp)$ are the short roots, each of multiplicity one, this forces
$b_1=b_{\mu_1}$ and $b_2=b_{\mu_2}$ for short roots $\mu_1>0>\mu_2$ with
$\mu_1+\mu_2=\wt y=\mu$. Now $\overline{T_-}(y)=b_\mu$ gives
$T_-(v_{\mu_1}\ot v_{\mu_2})\equiv v_\mu\pmod{qL(\vp)}$, and $\dim V_\mu=1$ lifts
the congruence to equality in (i) with $c_\mu\in 1+qA_0\subseteq  A_0^\times$.\\
For (ii), the zero-weight nodes are
neither the highest nor the lowest weight nodes of $C$, so Proposition~\ref{prop:twosided} gives their
factor weights $\nu_k>0$ and $-\nu_k<0$. The images
$\overline{T_-}(b_{\nu_k}\ot b_{-\nu_k})$ are the $m$ distinct zero-weight
nodes of $B(\vp)$, hence a basis of $L(\vp)_0/qL(\vp)_0$. Therefore, using Nakayama's lemma ($A_0$ is a local ring) it follows that $\{T_-(v_{\nu_k}\ot v_{-\nu_k})\}_{k=1}^m$ form an $A_0$-basis of $L(\vp)_0$, hence a
$\K$-basis of $V_0$.
\end{proof}

Let $\psi\colon V\ot_- V\xrightarrow{\sim}V\ot_+ V$ be Kashiwara's isomorphism
\cite[(1.4.5)]{Kashiwara}, $\psi(u\ot v)=q^{(\wt u,\wt v)}u\ot v$ on
weight vectors. Then $T:=T_-\circ\psi^{-1}\colon V\ot_+ V\to V$ is a surjective
$\Uq$-morphism for \eqref{eq:hopf}, and $T(v\ot w)=q^{-(\wt v,\wt w)}T_-(v\ot w)$ on
weight vectors. Consequently Theorem~\ref{thm:Tminus} holds for $T$ over $\K$. $T(v_{\mu_1}\ot v_{\mu_2})=q^{-(\mu_1,\mu_2)}c_\mu\,v_\mu, \;\;
q^{-(\mu_1,\mu_2)}c_\mu\in\K^\times$, and $(\nu_k,-\nu_k)=-2$, the vectors
$T(v_{\nu_k}\ot v_{-\nu_k})=q^{2}\,T_-(v_{\nu_k}\ot v_{-\nu_k})$, $k=1,2,\dots,m$, again form a basis of $V_0$.

\subsection{The braiding twist}\label{ss:braid}
For the course of the proof of the main result, we need to twist the morphism $T$ using the braid generator or braiding  $\Rhat$ of type-1 $\Uq$-modules. Let us quickly recall the construction of $\Rhat$ first. Let $\Theta=\sum_{\eta\in\Qp}\Theta_\eta$ be the quasi-$R$-matrix
\cite[Ch.~7]{Jantzen}, normalised by $\Theta_0=1\ot1$, with
$\Theta_\eta\in U^-_{-\eta}\ot U^+_{\eta}$. Write $f(v\ot w)=q^{-(\wt v,\wt
w)}v\ot w$ and $\tau(v\ot w)=w\ot v$. Then, $\Rhat=\Theta\circ f\circ\tau\colon V\ot_+ V\to V\ot_+ V$ is
a $\Uq$-module isomorphism for \eqref{eq:hopf}. For any two weight vectors $v$ and $w$,
\begin{equation}\label{eq:Rform}
\Rhat(v\ot w)=q^{-(\wt v,\wt w)}\,w\ot v
+\sum_{\eta>0}(u^-_{-\eta} w)\ot(u^+_\eta v), \;\;
\wt(u^-_{-\eta} w)<\wt w,\ \ \wt(u^+_\eta v)>\wt v.
\end{equation}
$\Rhat$ is an isomorphism and $[V\ot V:V]=1$, there is a scalar
$d\in\K^{\times}$ with
\begin{equation}\label{eq:TR}
T\circ\Rhat=d\, T.
\end{equation}
Recall from \S\ref{ss:threetypes} that $\alpha_i$, with $i=1,4,8$ in types $G_2$,
$F_4$, $E_8$ respectively, is the short simple root for which $\vp=\vp_i$. Thus
$z_{\alpha_i}$ is a zero-weight vector of $V$ and $b_{i,0}=\overline{z_{\alpha_i}}$.

\begin{proposition}\label{prop:d}
$d\neq1$.
\end{proposition}

\begin{proof}
Set $Z:=v_\vp\ot z_{\alpha_i}\in V\ot_+ V$. Assume the contrary, $d=1$. Then $T(Z)=T(\Rhat Z)$ by \eqref{eq:TR}, and we evaluate both sides using $T=T_-\circ\psi^{-1}$.

As $\wt z_{\alpha_i}=0$, $\psi^{-1}(Z)=Z$, so $T(Z)=T_-(Z)$. The image of $Z$ in $C$ is the highest weight node $x=b_\vp\ot b_{i,0}$. Hence $T_-(Z)\equiv v_\vp\pmod{qL(\vp)}$, and using $\dim V_\vp=1$, this congruence can be lifted to the equality,
\[
T(Z)=\kappa\,v_\vp,\qquad \kappa\in1+qA_0,\qquad T(Z)\not\equiv0\pmod{qL(\vp)}.
\]
It again follows from (\ref{eq:Rform}) that
%$\Rhat(Z)=q^{-(\vp,0)}z_{\alpha_i}\ot v_\vp+\sum_{\eta>0}(u^-_{-\eta}
%z_{\alpha_i})\ot(u^+_\eta v_\vp)$. The scalar equals $1$, and each correction vanishes since
%$u^+_\eta v_\vp=0$ for $\eta>0$, $v_\vp$ being highest. 
$\Rhat(Z)=z_{\alpha_i}\ot
v_\vp$. So, we must have $T(\Rhat Z)=T_-(z_{\alpha_i}\ot v_\vp)$.

The equality $T(Z)=T(\Rhat Z)$ now reads $T_-(v_\vp\ot z_{\alpha_i})=T_-(z_{\alpha_i}\ot v_\vp)$, so
$T_-(z_{\alpha_i}\ot v_\vp)\not\equiv0\pmod{qL(\vp)}$. But, the image of $z_{\alpha_i}\ot v_\vp$ in $C$ is $b_{i,0}\ot b_\vp$, we obtain
$\overline{T_-}(b_{i,0}\ot b_\vp)\neq0$. Therefore, $b_{i,0}\ot b_\vp \in C$. But $C$ contains exactly one node of weight $\vp$, that is $x=b_\vp\ot b_{i,0}$, and $b_{i,0}\ot b_\vp\neq x$. This contradicts $b_{i,0}\ot b_\vp\in C$. We conclude that $d\neq1$.
\end{proof}

\begin{remark}\label{rem:d}
In practice, $d$ can be computed explicitly. As $\Rhat^{\,2}$ is a $\Uq$-endomorphism
of $V\ot_+V$, it acts on each irreducible component by a scalar known as the monodromy eigenvalue. On the multiplicity free component of $V\ot_+ V$ isomorphic to $V$, this eigenvalue is exactly $d^2$, and it is of the form $d^{\,2}=q^{(\vp,\,\vp+2\rho)}$. Since $(\vp,\vp+2\rho)>0$, this gives $d^{\,2}\neq1$. For more details, we
refer to chapter~8 of \cite[Corollary 23]{KS} and \cite{Gould}. 
\end{remark}

The following twisted version of $T$ will be used in the next section.

\begin{corollary}\label{cor:extract}
\leavevmode
\begin{enumerate}
\item[\rm(i)] For $\mu\in W\vp$ with $\mu\neq\pm\vp$ and $\mu_1,\mu_2$ as in
Theorem~\ref{thm:Tminus}\,{\rm(i)},
\[
T\big(v_{\mu_1}\ot v_{\mu_2}-\Rhat(v_{\mu_1}\ot v_{\mu_2})\big)
=(1-d)\,T(v_{\mu_1}\ot v_{\mu_2})
=(1-d)\,q^{-(\mu_1,\mu_2)}c_\mu\,v_\mu
\]
is a nonzero scalar multiple of $v_\mu$.
\item[\rm(ii)] For $\nu_1,\dots,\nu_m$ as in Theorem~\ref{thm:Tminus}\,{\rm(ii)},
the vectors
\[
T\big(v_{\nu_k}\ot v_{-\nu_k}-\Rhat(v_{\nu_k}\ot v_{-\nu_k})\big)
=(1-d)\,T(v_{\nu_k}\ot v_{-\nu_k}),\qquad k=1,2,\dots,m,
\]
form a basis of $V_0$.
\end{enumerate}
\end{corollary}

\section{Generation via generalized quantum minors}\label{minors}

Put $I_s:=\{k\in I\mid\alpha_k\in\Pis\}$, so that $|I_s|=m$, and write
$z_k:=z_{\alpha_k}$ for $k\in I_s$. Let
$\{\,v_\mu^*\mid\mu\in\Phis\,\}\sqcup\{\,z_k^*\mid k\in I_s\,\}$
be the basis of $V^*$ dual to the weight basis \eqref{eq:qmbasis} of $V$. Therefore,
$v_\mu^*\in(V^*)_{-\mu}$ and $z_k^*\in(V^*)_0$. Set
\begin{gather*}
\mathscr{T}_{\neq0}:=\{\,c^V(v_\mu^*,v_{\mu'})\mid\mu,\mu'\in\Phis\,\},\\
\mathscr{T}_0:=\{\,c^V(z_k^*,v_\mu),\ c^V(v_\mu^*,z_k),\ c^V(z_k^*,z_l)
\mid\mu\in\Phis,\ k,l\in I_s\,\}.
\end{gather*}
Observe that $\mathscr{T}_{\neq0}=\{\,\Delta_{\mu,\mu'}\mid\mu,\mu'\in\Phis\,\}$
is the set of generalized quantum minors of $V$. Let $\mathscr{T}''\subseteq \mathcal{O}_q(G)$ be the
$\K$-subalgebra generated by $\mathscr{T}_{\neq0}$. Fix $r\in I_s$. Since $z_r^*$ has weight $0$, $T^*(z_r^*)$ must be of the form
\begin{equation}\label{eq:dagger}
T^*(z_r^*)=\underbrace{\sum_{k,l\in I_s} a^{(r)}_{kl}\,z_k^*\ot z_l^*}_{=:A_r}
\;+\;\underbrace{\sum_{\mu\in\Phis} n^{(r)}_{\mu}\,v_{\mu}^*\ot v_{-\mu}^*}_{=:N_r},
\qquad a^{(r)}_{kl},\,n^{(r)}_{\mu}\in\K .
\end{equation}

\begin{lemma}[{\rm cf.}\ {\rm\cite[Lemma~5.11]{OQY}}]\label{lem:sym}
The matrix $A^{(r)}=(a^{(r)}_{kl})_{k,l\in I_s}$ is symmetric for every $r\in I_s$.
\end{lemma}

\begin{proof}
For $G_2$, the matrix $A^{(r)}$ is a scalar, and the claim is trivial. For $\g=E_8$, $V=V(\vp)$
is the quantum adjoint module. Therefore, \cite[Lemma~5.11]{OQY} gives the symmetry of $A^{(r)}$.

For $F_4$, we have $\Pis=\{\alpha_3,\alpha_4\}$, so $I_s=\{3,4\}$. These two
short simple roots span a subsystem $\Phi(A_2)=\{\pm\alpha_3,\pm\alpha_4,
\pm\theta_{A_2}\}$ of type $A_2$, with $\theta_{A_2}:=\alpha_3+\alpha_4$. The generators $E_j,F_j,K_j^{\pm1}$ $(j\in\{3,4\})$ generate a subalgebra isomorphic to 
$U_q(\mathfrak{sl}_3)\subseteq\Uq$. Write $V':=V_0\oplus\bigoplus_{\mu\in\Phi(A_2)}V_\mu= \text{Span}_{\mathbb{K}}\{z_3, z_4\}\cup \{v_{\mu}\mid \mu \in \Phi(A_2)\}$. Using \emph{(Q1)}--\emph{(Q5)}, it follows that $V'$
is a $U_q(\mathfrak{sl}_3)$-submodule of $V$. It is isomorphic to the quantum adjoint module
$V_q(\theta_{A_2})$ with zero-weight space $V'_0=V_0$, and the set $\{z_3, z_4\}\cup \{v_{\mu}\mid \mu \in \Phi(A_2)\}$ is the canonical basis $B(\theta_{A_2})$. Now, $v_{\theta_{A_2}}$ and $v_{-\theta_{A_2}}$ are the highest and  lowest weight vectors
of $V'$, so $v_{\theta_{A_2}}\ot v_{-\theta_{A_2}}$ is a cyclic vector of the
$U_q(\mathfrak{sl}_3)$-module $V'\ot V'$. As $T$ preserves weights, it sends this cyclic
vector into $V_0= V'_0\subseteq V'$, and hence
$T(V'\ot V')=U_q(\mathfrak{sl}_3)\,T(v_{\theta_{A_2}}\ot v_{-\theta_{A_2}})\subseteq V'$. Thus $T$ restricts to a $U_q(\mathfrak{sl}_3)$-module homomorphism
$T|_{V'}\colon V'\ot V'\to V'$. By \eqref{eq:dagger} the matrix $A^{(r)}$ is the
$V_0^*\ot V_0^*$-component of $T^*(z_r^*)$, so it is the matrix attached to the pair
$\big(V',\;T|_{V'}\big)$. Applying \cite[Lemma~5.11]{OQY} again, we obtain $A^{(r)}$ is symmetric.
\end{proof}

\begin{lemma}[{\rm cf.}\ {\rm\cite[Theorem~5.9]{OQY}}]\label{lem:key}
Let $v,w\in V$ be weight vectors with $\wt v>0>\wt w$. Then 
$c^{V\ot V}\big(T^*(z_r^*),\ v\ot w-\Rhat(v\ot w)\big)\in\mathscr{T}'',
\;\; r\in I_s.$
\end{lemma}

\begin{proof}
Write $\zeta:=v\ot w-\Rhat(v\ot w)$ and split $T^*(z_r^*)=A_r+N_r$ as in
\eqref{eq:dagger}.

(i) By \eqref{eq:Rform}, we have $\wt(u^-_{-\eta}w)<\wt w<0$ and $\wt(u^+_\eta v)>\wt v>0$. Hence, in $v\ot w$ and in every summand of $\Rhat(v\ot w)$ both tensor slots carry non-zero weight.

(ii) Each summand of $N_r$ has both slots of non-zero weight, so using
\eqref{eq:mult} and (i), $c^{V\ot V}(N_r, \zeta)$ can be written a sum of products of
elements of $\mathscr{T}''$.

(iii) For $X\in\Uq$, using $\Rhat\,\Delta(X)=\Delta(X)\Rhat$ and transporting
$\Rhat=\Theta\circ f\circ\tau$ to the dual side,
\[
\big\langle A_r,\ \Delta(X)\Rhat(v\ot w)\big\rangle
=\sum_{\eta\ge0}q^{(\eta,\eta)}\big\langle
\tau^*\big((S^{-1}\ot S^{-1})(\Theta_\eta^{21})A_r\big),\ \Delta(X)(v\ot w)\big\rangle.
\]
Notice that the flip $^{21}$ appears in the above equation because we are using the opposite convention for the pairing in (\ref{eq:mult}). By definition, $\Theta_0=1\ot1$, and Lemma~\ref{lem:sym} gives $\tau^*A_r=\sum_{k,l\in I_s}a^{(r)}_{lk}z_k^*\ot z_l^*=A_r$, the $\eta=0$ term equals
$\big\langle A_r,\Delta(X)(v\ot w)\big\rangle$ so it cancels against the leading term.
Therefore
\[
c^{V\ot V}(A_r,\zeta)=-\sum_{\eta>0}q^{(\eta,\eta)}\,
c^{V\ot V}\big(\tau^*\big((S^{-1}\ot S^{-1})(\Theta_\eta^{21})A_r\big),\ v\ot w\big)
\ \in\ \mathscr{T}''.
\]
Combining the steps (ii) and (iii), the proof follows.
\end{proof}

\subsection*{The main results}
\begin{theorem}\label{thm:main}
Let $\g$ be of type $G_2$, $F_4$ or $E_8$ and $V=V(\vp)$ the quasi-minuscule module.
Then $\mathscr{T}_0\subseteq\mathscr{T}''$.
\end{theorem}

\begin{proof}
Step 1 ($W.\vp \setminus \{\pm \vp\}$). Let $\mu\in W\vp$ with $\mu\neq\pm\vp$, and let
$\mu_1,\mu_2$ be as in Theorem~\ref{thm:Tminus}\,(i). By
Corollary~\ref{cor:extract}(i),
$T\big(v_{\mu_1}\ot v_{\mu_2}-\Rhat(v_{\mu_1}\ot v_{\mu_2})\big)
=(1-d)\,q^{-(\mu_1,\mu_2)}c_\mu\,v_\mu$. The scalar $(1-d)q^{-(\mu_1,\mu_2)}c_\mu$ is non-zero, and $\mu_1>0>\mu_2$. Hence
\[
c^V(z_r^*,v_\mu)=\big((1-d)q^{-(\mu_1,\mu_2)}c_\mu\big)^{-1}
c^{V\ot V}\big(T^*(z_r^*),\ v_{\mu_1}\ot v_{\mu_2}-\Rhat(v_{\mu_1}\ot v_{\mu_2})\big)
\in\mathscr{T}''
\]
by Lemma~\ref{lem:key}, for every $r\in I_s$ and every short root $\mu\neq\pm\vp$.

Step 2 ($\vp\;\text{and} -\vp$). We show that $c^V(z_r^*,v_{\pm\vp})\in\mathscr{T}''$ for
all $r\in I_s$. By the construction of $T$, and the description of $x$ and
$x_{\mathrm{low}}$ of Proposition~\ref{prop:twosided}, we obtain the relations
\begin{equation}\label{eq:step2rel}
T(v_\vp\ot z_i)=\kappa_+\,v_\vp,\qquad
T(z_i\ot v_{-\vp})=\kappa_-\,v_{-\vp},\qquad \kappa_\pm\in\K^\times.
\end{equation}
Applying $F_iE_i$ to the first relation and using \emph{(Q1)}--\emph{(Q5)}, a direct computation shows that
\begin{equation}\label{eq:step2key}
T(v_{\vp-\alpha_i}\ot v_{\alpha_i})=\kappa'\,v_\vp,\qquad \kappa'\in\K^\times.
\end{equation}
Here, both tensor slots carry positive weight, $\vp-\alpha_i>0$ and $\alpha_i>0$. Using \eqref{eq:step2key} with $T^*(z_r^*)=A_r+N_r$ gives
\[
\kappa'\,c^V(z_r^*,v_\vp)=c^{V\ot V}\big(N_r,\,v_{\vp-\alpha_i}\ot v_{\alpha_i}\big)
+\sum_{k,l\in I_s}a^{(r)}_{kl}\,c^V(z_l^*,v_{\vp-\alpha_i})\,c^V(z_k^*,v_{\alpha_i}).
\]
The first term is in $\mathscr{T}''$. In the
second term, each factor has a zero-weight functional and an extremal vector with weight
$\vp-\alpha_i,\alpha_i\in W\vp\setminus\{\pm\vp\}$, hence lies in $\mathscr{T}''$ by
Step~1. Thus $c^V(z_r^*,v_\vp)\in\mathscr{T}''$. Applying $E_iF_i$ to the
second relation of \eqref{eq:step2rel} gives
$T(v_{-\alpha_i}\ot v_{-\vp+\alpha_i})=\kappa''\,v_{-\vp}$ with $\kappa''\in\K^\times$
and both slots of negative weight. So, $c^V(z_r^*,v_{-\vp})\in\mathscr{T}''$ by a similar argument. Together with Step~1, this gives
$c^V(z_r^*,v_\mu)\in\mathscr{T}''$ for all $\mu\in\Phis$ and all $r\in I_s$.

Step 3 (zero-weight vectors). Let $\nu_1,\dots,\nu_m$ be as in
Theorem~\ref{thm:Tminus}\,(ii). By Corollary~\ref{cor:extract}(ii) the vectors
$T\big(v_{\nu_k}\ot v_{-\nu_k}-\Rhat(v_{\nu_k}\ot v_{-\nu_k})\big)$,
$k=1,2,\dots,m$, form a basis of $V_0$. Writing
$z_j=\sum_{k=1}^{m}\lambda^{(j)}_k\,
T\big(v_{\nu_k}\ot v_{-\nu_k}-\Rhat(v_{\nu_k}\ot v_{-\nu_k})\big)$ with
$\lambda^{(j)}_k\in\K$, we get 
$c^V(z_r^*,z_j)=\sum_{k=1}^{m}\lambda^{(j)}_k\,
c^{V\ot V}\big(T^*(z_r^*),\ v_{\nu_k}\ot v_{-\nu_k}
-\Rhat(v_{\nu_k}\ot v_{-\nu_k})\big)\in\mathscr{T}''$
by Lemma~\ref{lem:key}, for all $r,j\in I_s$.

Step 4 (transposed family). Let $\mu\in\Phis$ and $j\in I_s$. Similarly following step 3, write $c^V(v_\mu^*,z_j)=\sum_{k=1}^{m}\lambda^{(j)}_k(1-d)\,
c^{V\ot V}\big(T^*(v_\mu^*),\,v_{\nu_k}\ot v_{-\nu_k}\big).$ Therefore, using Step~1, and Step~2 again, we immediately conclude that $c^V(v_\mu^*,z_j)\in\mathscr{T}''$. 

%Decompose $T^*(v_\mu^*)$ into weight components as in \eqref{eq:dagger}. Since $\wt(v_\mu^*)=-\mu\neq0$,
%no component has both slots of weight $0$. Those with both slots of non-zero weight
%give, by \eqref{eq:mult}, products of elements of $\mathscr{T}_{\neq0}$. A component
%with exactly one zero slot has the form $z_l^*\ot\xi$ or %$\xi\ot z_l^*$ with
%$\wt\xi\neq0$, and \eqref{eq:mult} turns it into
%$c^V(\xi,v_{\nu_k})\,c^V(z_l^*,v_{-\nu_k})$ or
%$c^V(z_l^*,v_{\nu_k})\,c^V(\xi,v_{-\nu_k})$; in each case the factors lie in
%$\mathscr{T}''$ either directly or by Steps~1--2, since
%$\pm\nu_k\in\Phis$.

Combining steps 1--4 exhausts $\mathscr{T}_0$, and this settles the proof.
\end{proof}

\begin{corollary}\label{cor:allG}
Let $G$ be a simply connected simple complex algebraic group. Then $\mathcal{O}_q(G)$
is generated as a $\K$-algebra by generalized quantum minors.
\end{corollary}

\begin{proof}
If $\g=Lie(G)$ is of type $A_n$, $B_n$, $C_n$, $D_n$,
$E_6$ or $E_7$, then $\g$ admits a family of non-trivial
minuscule modules that is also a tensor generating family, so $\mathcal{O}_q(G)$ is generated
by their matrix coefficients. Every weight of a minuscule module is extremal, so each of these matrix coefficients is a generalized quantum minor. Hence, the assertion follows.

In the remaining types $G_2$, $F_4$, and $E_8$,
the quasi-minuscule module $V=V(\vp)$ is a tensor generator. So $\mathcal{O}_q(G)$ is
generated by $\mathscr{T}_{\neq0}\sqcup\mathscr{T}_0$. By Theorem~\ref{thm:main} we have
$\mathscr{T}_0\subseteq\mathscr{T}''$. Hence, $\mathcal{O}_q(G)=\mathscr{T}''$ is
generated by $\mathscr{T}_{\neq0}$, which is the set of generalized
quantum minors of $V$.
\end{proof}

\begin{remark}
\label{srings}
Corollary~\ref{cor:allG} is stated over $\K=\mathbb Q(q)$. In fact \cite{OQY}
establish a sharp refinement of it over much smaller ground rings. For types $A_n$--$D_n$, $E_6$, and $E_7$, the corresponding statement holds over the
Laurent ring $\mathcal A=\mathbb Z[q^{\pm1/2}]$, and for $G_2$ over
$\mathcal A\big[(q^2+1)^{-1}\big]$. However, the $E_8$ case is still obtained only
over $\mathbb Q(q^{1/2})$, and the $F_4$ case was left open.
\end{remark}

\begin{corollary}\label{cor:cluster}
Let $G$ be of type $F_4$. Then, $\mathcal{O}_q(G)$ admits a quantized cluster
algebra structure over $\mathbb Q(q^{1/2})$.
\end{corollary}

\begin{proof}
Corollary~\ref{cor:allG} gives in type $F_4$ the generation of $\mathcal{O}_q(G)$ via generalized quantum minors. It immediately implies that $\mathcal{O}_q(G)$ admits a quantized cluster algebra structure over $\mathbb{Q}(q^{1/2})$ by the same argument given before
\cite[Theorem~5.4]{OQY}.
\end{proof}

\begin{remark}\label{obs}
It is worthwhile to mention that Corollary \ref{cor:allG} for $G=F_4$ cannot be specialized at $q=1$. To be precise, the quasi-minuscule module $V(\vp_4)$ appears in the symmetric subspace of the tensor product module $V(\vp_4)\otimes V(\vp_4)$, we have $d=1$ ($\hat{R}$ is replaced by the flip)  for $\mathcal{O}(F_4)$ (see Remark \ref{rem:d}). Therefore, the factor $\frac{1}{1-d}$ repeatedly used in the proof of the Theorem \ref{thm:main} blows up as $q\to 1$. In \cite[Remark 3.10]{Oya}, Oya explained why his techniques also could not touch the $F_4$ case in the classical setting. At this point, the most promising approach to settle this problem seems to reprove the Corollary \ref{cor:allG} for $F_4$ over $\mathcal{A}$, and then specialize at $q=1$. But we do not have a proof of this yet. 
\end{remark}

\medskip 

\noindent \textbf{Acknowledgments}: The problem in type $F_4$ was proposed to the author by Hironori Oya. The author is deeply grateful to him for many valuable discussions, thoughtful comments, and for pointing out exactly what is unresolved and what to look for. The author would also like to thank his thesis advisor, Prof. Arup K. Pal, for discussions regarding the braiding and Proposition \ref{prop:d}. The author also acknowledges financial support from the Indian Statistical Institute through a PhD fellowship.

\end{document}